\documentclass[11pt]{article}

\usepackage[margin=1in]{geometry}
\usepackage[T1]{fontenc}
\usepackage[utf8]{inputenc}
\usepackage{lmodern}
\usepackage{amsmath,amssymb,amsthm,mathtools}
\usepackage{enumitem}
\usepackage{hyperref}
\usepackage{microtype}

\hypersetup{
  colorlinks=true,
  linkcolor=blue,
  citecolor=blue,
  urlcolor=blue
}

\numberwithin{equation}{section}

\newtheorem{theorem}{Theorem}[section]
\newtheorem{lemma}[theorem]{Lemma}

\newtheorem{definition}[theorem]{Definition}
\newtheorem{assumption}[theorem]{Assumption}
\newtheorem{remark}[theorem]{Remark}

\newcommand{\dd}{\,\mathrm{d}}
\newcommand{\eps}{\varepsilon}
\newcommand{\R}{\mathbb{R}}
\newcommand{\E}{\mathbb{E}}
\newcommand{\Prob}{\mathbb{P}}
\newcommand{\X}{\mathcal{X}}
\newcommand{\Xo}{\mathcal{X}^{0}}
\newcommand{\cS}{\mathcal{S}}
\newcommand{\cP}{\mathcal{P}}

\def\({\left(}\def\){\right)}

\title{Strong uniqueness and large deviation principle for\\ mutually catalytic super Markov chains \thanks{This work was supported by
		the National Key R\&D Program of China (2022YFA1006102), and
		the National Natural Science Foundation of China (12471418, 12595294, 12231002), and  the New Cornerstone Science Foundation(NCI202501).}}
\author{Jie Xiong\thanks{Department of Mathematics and Shenzhen International Center for Mathematics, Southern University of Science and Technology, Shenzhen 518055, China. Email: \texttt{xiongj@sustech.edu.cn}.}
\and Wen Xu\thanks{School of Mathematical Sciences, Peking University, Beijing 100871, China. Email: \texttt{xuwen@math.pku.edu.cn}.}}

\date{}

\begin{document}
\maketitle

\begin{abstract}
In this paper, we study the strong uniqueness problem for the mutually catalytic super-Markov chain, which is a two-dimensional degenerate stochastic differential equation with H\"older continuous coefficients.  The key step is to find a process which is a function of two coupled processes and satisfies an autonomous one-dimensional stochastic differential equation; uniqueness for this equation follows from a Yamada--Watanabe argument.  A large deviation principle is then established, in the irreducible two-state case, by applying the weak-convergence approach of Budhiraja, Dupuis and Maroulas to the controlled equations.
\end{abstract}


\section{Introduction}

Super-processes are measured-valued processes arising from high density limit of branching particle systems. We refer the reader to Dawson \cite{Daw} and Watanabe \cite{Wat} for the origins of these processes. Independently, Konno and Shiga \cite{KS} and Reimers \cite{Rei} proved that the superprocess $\mu_t$ admits a density, i.e. $\mu_t(\dd x)=v_t(x)\dd x$, where $v_t(x)$ is the unique weak solution to the stochastic partial differential equation
\begin{equation}\label{eq:sbm-spde}
 \partial_t v_t(x)=\frac12\Delta v_t(x)+\sqrt{\lambda v_t(x)}\,\dot B_{t,x},\quad v_0=\frac{\mu_0(\dd x)}{\dd x}.
\end{equation}
Here $\dot B_{t,x}$ denotes space-time white noise. The strong uniqueness problem for \eqref{eq:sbm-spde} has attracted considerable attention
(cf. Mytnik et al. \cite{MPS}, Mytnik and Perkins \cite{MP}, Burdzy et al. \cite{BMP}, Mueller et al. \cite{MMP}, and Xiong \cite{Xio}).

When the branching rate depends on the concentration of another super-Brownian population, the density process $v_t$ satisfies
\begin{equation}\label{eq:catalyst-reactant}
 \partial_t v_t(x)=\frac12\Delta v_t(x)+\sqrt{u_t(x)v_t(x)}\,\dot B_{t,x},\quad v_0=\mu,
\end{equation}
where $u_t(x)$ is the density of another super-Brownian motion independent of the noise $\dot B_{t,x}$.  In this setting, $u_t(x)$ is called the catalyst and $v_t(x)$ the reactant.  The resulting process is known as a super-Brownian motion with a super-Brownian catalyst; see Dawson and Fleischmann \cite{DF}.

A further study is the mutually catalytic branching model
\begin{equation}\label{eq:mutual-spde}
\begin{cases}
 \partial_t v_t(x)=\frac12\Delta v_t(x)+\lambda_1\sqrt{u_t(x)v_t(x)}\,\dot B_{t,x},\\[2mm]
 \partial_t u_t(x)=\frac12\Delta u_t(x)+\lambda_2\sqrt{u_t(x)v_t(x)}\,\dot W_{t,x},
\end{cases}
\end{equation}
where $\dot W_{t,x}$ is independent of $\dot B_{t,x}$.  In the symmetric case $\lambda_1=\lambda_2$, Dawson and Perkins \cite{DP} established weak uniqueness using self-duality.  

A natural extension is the cyclically catalytic model
\begin{equation}\label{eq:cyclic-spde}
 \partial_t u_t^k(x)=\frac12\Delta u_t^k(x)+\sqrt{u_t^k(x)u_t^{k+1}(x)}\,\dot W_{t,x}^k,
 \quad k=1,2,\ldots,K,
\end{equation}
where $u_t^{K+1}=u_t^1$.  Although weak uniqueness remains open for this model, Fleischmann and Xiong \cite{FX} constructed a strong Markov solution through a Markov selection argument and studied its long-time segregation and extinction-survival properties.

This motivates the uniqueness problem studied in this paper.  More precisely, we consider the case where the underlying spatial motion is collapsed to a single point, and study uniqueness for the $K$-dimensional stochastic differential equation
\begin{equation}\label{eq:K-sde}
 \dd X_t^k=\sum_{j=1}^K q_{kj}X_t^j\dd t+\sqrt{\gamma_k X_t^kX_t^{k+1}}\,\dd W_t^k,
 \quad k=1,2,\ldots,K,
\end{equation}
where $q_{ij}\ge0$ for $i\ne j$, $\gamma_k>0$, and $W^1,\ldots,W^K$ are mutually independent standard Brownian motions.  In particular, if $(q_{ij})$ is the $Q$-matrix of a Markov chain, then \eqref{eq:K-sde} describes a super-Markov chain.

Athreya et al. \cite{ABBP} studied the SDE
\begin{equation}\label{eq:abbp-sde}
 \dd X_t^k=b_k(X_t)\dd t+\sqrt{X_t^k\gamma_k(X_t)}\,\dd W_t^k,
 \qquad k=1,2,\ldots,K.
\end{equation}
They proved weak uniqueness under assumptions that $\gamma_k(x)>0$ for all $x$, that $b_k(x)>0$ for $x\in\partial\R_+^K$, and that the $b_k(x)$ have linear growth.  These assumptions are not satisfied by \eqref{eq:K-sde}.  More general equations were studied by Bass and Perkins \cite{BP,BP2}, but \eqref{eq:K-sde} is still not covered by these extensions.  Dawson and Perkins \cite{DP2} obtained a further extension by considering
\begin{equation}\label{eq:dp2-sde}
 \dd X_t^k=b_k(X_t)\dd t+\sqrt{\gamma_k(X_t)X_t^kX_t^{k+1}}\,\dd W_t^k,
 \qquad k=1,2,\ldots,K.
\end{equation}
Under suitable H\"older continuity and boundary conditions on the coefficients, with $\gamma_k(x)$ strictly positive and locally bounded away from zero, they proved weak uniqueness for the associated martingale problem.

The weak uniqueness results above naturally lead to the stronger question of whether strong uniqueness holds for the degenerate catalytic system \eqref{eq:K-sde}.  In the case $q_{ij}=0$ for all $i\ne j$, Dawson et al. \cite{DFX} established strong uniqueness for \eqref{eq:K-sde}.  Their argument relies crucially on the special absorbing property that any component, once it reaches zero, remains at zero forever.  In the same spirit, with $\sqrt{\gamma_k X_t^kX_t^{k+1}}$ replaced by $\sqrt{X_t^k f_k(X_t)}$, the strong uniqueness problem was studied by He \cite{He}.  The goal of this article is to establish strong uniqueness for \eqref{eq:K-sde} in the two-dimensional mutually catalytic case.  The general case will be discussed in a future work. 
We point out that strong uniqueness results for multidimensional SDEs with non-Lipschitz coefficients remain limited.  Apart from the works of DeBlassie \cite{Deb} and Swart \cite{Swa2}, few results seem to be available in this direction.  In particular, DeBlassie \cite{Deb} posed the strong uniqueness problem for \eqref{eq:abbp-sde} as an open problem.

We also investigate the small-noise asymptotic behavior of the two-dimensional mutually catalytic system. Large deviation principles (LDP) for small-noise SDEs originate from the classical Freidlin--Wentzell theory; see, for example, Freidlin and Wentzell \cite{FreidlinWentzell}. Extensions to positive diffusions with coefficients degenerating at the boundary have been studied in \cite{BC}. In the present model, however, the diffusion coefficient
$
(x_1,x_2)\longmapsto \sqrt{x_1x_2}
$
is degenerate on the boundary of the positive quadrant and is only H{\"o}lder continuous, so the standard Freidlin--Wentzell theory cannot be applied directly without additional analysis of the boundary behavior. We therefore employ the weak-convergence approach developed by Budhiraja et al. \cite{BDM}. The main ingredients are the well-posedness of the associated controlled skeleton equation, the uniform separation of finite-energy skeleton paths from the boundary in the irreducible case, the compactness of the corresponding controlled trajectories, and the convergence of the controlled stochastic equations to the deterministic skeleton. This yields an explicit action functional describing the quadratic control energy required to produce a prescribed trajectory.

The main contributions are as follows. 
 Our first result is the strong uniqueness of the solution to \eqref{eq:K-sde} for $K=2$. Our second result establishes a LDP for the same system when the branching rates are multiplied by a small parameter $\eps$.

The rest of this article is organized as follows.  Section \ref{sec:main} states the main results.  Section \ref{sec:strong-uniqueness} proves strong uniqueness in the two-dimensional case.  Section \ref{sec:ldp} establishes the LDP for the mutually catalytic Markov chain with small branching rates.  Appendix \ref{app:bdm} recalls the weak-convergence framework of Budhiraja et al. \cite{BDM}.

\section{Main results}\label{sec:main}

Let $\X=[0,\infty)^2$ and $\Xo=(0,\infty)^2$.  In the two-dimensional case, \eqref{eq:K-sde} becomes
\begin{equation}\label{eq:main-sde}
 \dd X_t^k=\sum_{j=1}^2 q_{kj}X_t^j\dd t+\sqrt{\gamma_kX_t^1X_t^2}\,\dd W_t^k,
 \quad k=1,2.
\end{equation}
Here $q_{12},q_{21}\ge0$, $q_{11},q_{22}\in\R$, and $\gamma_1,\gamma_2>0$.  In the super-Markov-chain case one usually has $q_{11}=-q_{12}$ and $q_{22}=-q_{21}$. The proof below does not require these identities.

Weak existence and uniqueness in law for equation \eqref{eq:main-sde} are already available from the existing weak-uniqueness theory for catalytic branching diffusions; see, for instance, \cite{DP2}. Therefore, by the Yamada--Watanabe theorem, in order to establish strong well-posedness it suffices to prove pathwise uniqueness. 

\begin{theorem}[Pathwise uniqueness]\label{thm:pathwise-uniqueness}
Suppose that $x_0=(x_0^1,x_0^2)\in\Xo$.  Then pathwise uniqueness holds for \eqref{eq:main-sde}.
\end{theorem}

We next consider the LDP for the small-noise mutually catalytic Markov chain
\begin{equation}\label{eq:small-noise-x}
 \dd X_t^{\eps,k}=\sum_{j=1}^2 q_{kj}X_t^{\eps,j}\dd t+\sqrt{\eps\gamma_k X_t^{\eps,1}X_t^{\eps,2}}\,\dd W_t^k,
 \quad k=1,2,
\end{equation}
with $X_0^\eps=x_0\in\Xo$.

\begin{theorem}[Large deviation principle]\label{thm:ldp}
Assume that $q_{12},q_{21}>0$.  Then $\{X^\eps\}_{\eps>0}$ satisfies an LDP on $C([0,T];\X)$ with rate function
\begin{equation}\label{eq:rate-function}
 I(x)=
 \begin{cases}
 \displaystyle
 \frac12\int_0^T\sum_{k=1}^2
 \frac{\left|\dot x_t^k-\sum_{j=1}^2q_{kj}x_t^j\right|^2}
 {\gamma_kx_t^1x_t^2}\dd t,
 &\begin{array}{l}
 x\in AC([0,T];\Xo),\\ x(0)=x_0,
 \end{array}\\[7mm]
 \infty,&\text{otherwise.}
 \end{cases}
\end{equation}
Here $AC([0,T];\Xo)$ denotes the set of absolutely continuous paths taking values in $\Xo$.  
\end{theorem}

\begin{remark}
The integral expression in \eqref{eq:rate-function} is used only for paths that remain in the interior $\Xo$ throughout the time interval. In particular, paths with an initial value different from $x_0$, as well as paths that touch the boundary of $\X$, have infinite cost.
	The strict positivity of $q_{12}$ and $q_{21}$ in Theorem \ref{thm:ldp} is used to keep the controlled skeleton paths away from the boundary on finite-energy level sets.  Reducible cases, for instance $q_{12}=0$ or $q_{21}=0$, require an additional boundary convention and are not treated here.
\end{remark}

\section{Strong uniqueness}\label{sec:strong-uniqueness}

The main idea is to construct a one-dimensional process $P_t$ related to $(X_t^1,X_t^2)$ that satisfies an autonomous SDE.  Once $P_t$ is determined, the sum process $S_t=X_t^1+X_t^2$ satisfies a linear SDE.  The mapping
\[
 (X_t^1,X_t^2)\longleftrightarrow (P_t,S_t),\quad
 P_t=\frac{X_t^1}{X_t^1+X_t^2},\quad S_t=X_t^1+X_t^2,
\]
is bijective whenever $S_t>0$.

Let
\[
 \tau=\inf\{t\ge0:X_t^1+X_t^2=0\}.
\]
If $\tau<\infty$ for a path, it is clear that the process will remain at $(0,0)$ after that. Therefore, we only need to prove pathwise uniqueness
	before time $\tau$. 
On $[0,\tau)$ define $P_t=X_t^1/(X_t^1+X_t^2)$ and $S_t=X_t^1+X_t^2$.

\begin{lemma}\label{lem:P-equation}
The process $P_t$ satisfies, on $[0,\tau)$,
\begin{equation}\label{eq:P-equation}
\begin{aligned}
 \dd P_t={}&b(P_t)\dd t+
 \sqrt{P_t(1-P_t)}
 \left(\sqrt{\gamma_1}(1-P_t)\dd W_t^1-\sqrt{\gamma_2}P_t\dd W_t^2\right),
\end{aligned}
\end{equation}
where
\begin{equation}\label{eq:b-with-ito}
 b(p)=q_{12}(1-p)^2+(q_{11}-q_{22})p(1-p)-q_{21}p^2
 -\gamma_1p(1-p)^2+\gamma_2p^2(1-p).
\end{equation}
Moreover, pathwise uniqueness holds for \eqref{eq:P-equation} among $[0,1]$-valued solutions.
\end{lemma}

\begin{proof}

 After some computations based on It\^o's formula, the drift contribution from the linear terms as
\begin{align*}
 &(1-p)\{q_{11}p+q_{12}(1-p)\}-p\{q_{21}p+q_{22}(1-p)\} \\
 &\quad =q_{12}(1-p)^2+(q_{11}-q_{22})p(1-p)-q_{21}p^2,
\end{align*}
and the It\^o correction as
$
 -\gamma_1p(1-p)^2+\gamma_2p^2(1-p).
$
The martingale part is exactly the martingale part in \eqref{eq:P-equation}.

It remains to prove pathwise uniqueness for \eqref{eq:P-equation}.  Denote
\[
 \sigma_1(p)=\sqrt{\gamma_1}\sqrt{p(1-p)}(1-p),\quad
 \sigma_2(p)=-\sqrt{\gamma_2}\sqrt{p(1-p)}p.
\]
Since $b$ is a polynomial on $[0,1]$, it is Lipschitz there. Thus there exists $L_b>0$ such that
\[
 |b(p)-b(q)|\le L_b|p-q|,
 \quad p,q\in[0,1].
\]
Moreover,
\[
 a_1(p):=\sigma_1(p)^2=\gamma_1p(1-p)^3,
 \quad
 a_2(p):=\sigma_2(p)^2=\gamma_2p^3(1-p)
\]
are $C^1$ functions on $[0,1]$.  Hence, for some finite constant $L_\sigma$,
\[
 |a_1(p)-a_1(q)|+|a_2(p)-a_2(q)|\le L_\sigma|p-q|,
 \quad p,q\in[0,1].
\]
Since $(\sqrt u-\sqrt v)^2\le |u-v|$ for $u,v\ge0$, we obtain
\begin{equation}\label{eq:sigma-yw-bound}
 |\sigma_1(p)-\sigma_1(q)|^2+|\sigma_2(p)-\sigma_2(q)|^2\le L_\sigma|p-q|,
 \quad p,q\in[0,1].
\end{equation}

Let $P_t$ and $\widetilde P_t$ be two $[0,1]$-valued solutions driven by the same Brownian motions and with the same initial value.  Set $\Delta_t=P_t-\widetilde P_t$.  Let $(\varphi_n)_{n\ge1}\subset C^2(\R)$ be the standard Yamada--Watanabe convex approximation of $|x|$, with $\varphi_n(x)\uparrow |x|$, $|\varphi_n'(x)|\le1$, and
\[
 0\le\varphi_n''(x)\le \frac{2}{n|x|},\quad x\ne0.
\]
Applying It\^o's formula to $\varphi_n(\Delta_t)$ and using \eqref{eq:sigma-yw-bound}, we obtain
\begin{align*}
 \E[\varphi_n(\Delta_{t\wedge\tau})]
  &= \mathbb{E} \int_0^{t\wedge\tau} \varphi_n^{\prime}\left(\Delta_s\right)\left(b\left(P_s\right)-b\left(\widetilde{P}_s\right)\right) \mathrm{d} s \\
 & \quad \quad +\frac{1}{2} \mathbb{E} \int_0^{t\wedge\tau} \varphi_n^{\prime \prime}\left(\Delta_s\right)\left(\left|\sigma_1\left(P_s\right)-\sigma_1\left(\widetilde{P}_s\right)\right|^2+\left|\sigma_2\left(P_s\right)-\sigma_2\left(\widetilde{P}_s\right)\right|^2\right) \mathrm{d} s\\
  &\le L_b\int_0^t\E|\Delta_{s\wedge\tau}|\dd s
 +\frac12L_\sigma\E\int_0^{t\wedge\tau}\varphi_n''(\Delta_s)|\Delta_s|\dd s \\
 &\le L_b\int_0^t\E|\Delta_{s\wedge\tau}|\dd s+\frac{L_\sigma t}{n}.
\end{align*}
Letting $n\to\infty$ and using monotone convergence gives
\[
 \E|\Delta_{t\wedge\tau}|\le L_b\int_0^t\E|\Delta_{s\wedge\tau}|\dd s.
\]
Gronwall's lemma implies $\E|\Delta_{t\wedge\tau}|=0$ for all $t\le T$.  Since $T$ is arbitrary, pathwise uniqueness follows.
\end{proof}

\begin{lemma}\label{lem:S-equation}
On $[0,\tau)$, the process $S_t=X_t^1+X_t^2$ satisfies
\begin{equation}\label{eq:S-equation}
 \dd S_t=S_tA(P_t)\dd t+S_t\sqrt{P_t(1-P_t)}
 \left(\sqrt{\gamma_1}\dd W_t^1+\sqrt{\gamma_2}\dd W_t^2\right),
\end{equation}
where
\begin{equation}\label{eq:A-def}
 A(p)=(q_{11}+q_{21})p+(q_{12}+q_{22})(1-p).
\end{equation}
Once $P_t$ is fixed, \eqref{eq:S-equation} has a unique strong solution, given by the stochastic exponential
\begin{align}\label{eq:S-explicit}
 S_t=s_0\exp\Bigg\{&\int_0^t\left[A(P_r)-\frac12(\gamma_1+\gamma_2)P_r(1-P_r)\right]\dd r\nonumber \\
 &\qquad +\int_0^t\sqrt{P_r(1-P_r)}
 \left(\sqrt{\gamma_1}\dd W_r^1+\sqrt{\gamma_2}\dd W_r^2\right)\Bigg\},
\end{align}
where $s_0=x_0^1+x_0^2$.
\end{lemma}

\begin{proof}
Adding the two equations in \eqref{eq:main-sde} gives \eqref{eq:S-equation}.  For a fixed path $P_t$, this is a linear SDE with continuous adapted coefficients. Hence, it has a unique strong solution up to time $\tau$.  Formula \eqref{eq:S-explicit} is the usual Doleans-Dade's exponential representation, again up to time $\tau$. If $\tau<\infty$, then
\[\lim_{t\to \tau-}S_t=s_0\exp\(\mbox{a finite random variable}\)>0.\]
Thus, $\tau=\infty$ a.s. and  $S_t>0$ for every finite $t$.
\end{proof}

\begin{proof}[Proof of Theorem \ref{thm:pathwise-uniqueness}]
	
Let $X_t$ and $\widetilde X_t$ be two solutions of \eqref{eq:main-sde} with the same initial condition and driven by the same Brownian motion.  Define $S_t=X^1_t+X^2_t$ and $\widetilde S_t=\widetilde X^1_t+\widetilde X^2_t$.  By Lemma \ref{lem:S-equation}, $\tau=\tilde{\tau}=\infty$ a.s.

The associated processes $P_t$ and $\widetilde P_t$ are $[0,1]$-valued solutions of \eqref{eq:P-equation} with the same initial value.  Lemma \ref{lem:P-equation} gives $P_t=\widetilde P_t$ a.s.  Then Lemma \ref{lem:S-equation} gives $S_t=\widetilde S_t$ a.s.  We obtain,
\[
 X_t^1=P_tS_t,\quad X_t^2=(1-P_t)S_t,
\]
and the same identities hold for $\widetilde X_t$.  Hence $X_t=\widetilde X_t$ a.s.
\end{proof}

\section{Large deviation principle}\label{sec:ldp}

Throughout this section we assume
$
 q_{12},q_{21}>0,
$
and $x_0\in\Xo$.  Define
\[
 p_0=\frac{x_0^1}{x_0^1+x_0^2},\quad s_0=x_0^1+x_0^2.
\]
For $X^\eps_t$ solving \eqref{eq:small-noise-x}, set
\[
 P_t^\eps=\frac{X_t^{\eps,1}}{X_t^{\eps,1}+X_t^{\eps,2}},\quad
 S_t^\eps=X_t^{\eps,1}+X_t^{\eps,2}.
\]
As in Section \ref{sec:strong-uniqueness}, $S^\eps_t$ remains strictly positive on finite time intervals.

Let
\begin{align*}
 b_0(p)&=q_{12}(1-p)^2+(q_{11}-q_{22})p(1-p)-q_{21}p^2,\\
 b_1(p)&=-\gamma_1p(1-p)^2+\gamma_2p^2(1-p),
\end{align*}
and let $A$ be defined by \eqref{eq:A-def}.   Also define

\[
\begin{aligned}
	\alpha_1(p)
	&=\sqrt{\gamma_1}\sqrt{p(1-p)}(1-p),
	&\quad
	\alpha_2(p)
	&=-\sqrt{\gamma_2}\sqrt{p(1-p)}p,
	\\[-2mm]
	\eta_1(p)
	&=\sqrt{\gamma_1}\sqrt{p(1-p)},
	&
	\eta_2(p)
	&=\sqrt{\gamma_2}\sqrt{p(1-p)}.
\end{aligned}
\]

Then $P^\eps_t$ and $S^\eps_t$ satisfy the  equations
\begin{equation}\label{eq:Peps-correct}
 \dd P_t^\eps=\{b_0(P_t^\eps)+\eps b_1(P_t^\eps)\}\dd t
 +\sqrt{\eps}\left\{\alpha_1(P_t^\eps)\dd W_t^1+\alpha_2(P_t^\eps)\dd W_t^2\right\},
\end{equation}
\begin{equation}\label{eq:Seps-correct}
 \dd S_t^\eps=S_t^\eps A(P_t^\eps)\dd t
 +\sqrt{\eps} S_t^\eps\left\{\eta_1(P_t^\eps)\dd W_t^1+\eta_2(P_t^\eps)\dd W_t^2\right\}.
\end{equation}

Let $u=(u^1,u^2)\in L^2([0,T];\R^2)$.  The controlled skeleton associated with \eqref{eq:Peps-correct}--\eqref{eq:Seps-correct} is
\begin{equation}\label{eq:p-skeleton}
 \dot p_t=b_0(p_t)+\alpha_1(p_t)u_t^1+\alpha_2(p_t)u_t^2,
 \quad p(0)=p_0,
\end{equation}
and
\begin{equation}\label{eq:s-skeleton}
 \dot s_t=s_t\left\{A(p_t)+\eta_1(p_t)u_t^1+\eta_2(p_t)u_t^2\right\},
 \quad s(0)=s_0.
\end{equation}

\begin{lemma}[Well-posedness and boundary separation of the skeleton]\label{lem:skeleton-wellposed}
For every $u\in L^2([0,T];\R^2)$, equations \eqref{eq:p-skeleton}--\eqref{eq:s-skeleton} have a unique solution $(p_t,s_t)\in C([0,T];(0,1)\times(0,\infty))$.  Moreover, for every $N<\infty$, there are constants $0<\delta_N<1/2$ and $0<m_N<M_N<\infty$ such that, whenever
\[
 \int_0^T |u_t|^2\dd t\le N,
\]
one has
\begin{equation}\label{eq:uniform-boundary-separation}
 \delta_N\le p_t\le1-\delta_N,
 \quad
 m_N\le s_t\le M_N,
 \quad 0\le t\le T.
\end{equation}
\end{lemma}

\begin{proof}
We first prove that, on bounded control-energy sets, the trajectories
$p_t$ stay uniformly away from both boundaries.  For $p_t\in(0,1)$,
\begin{align*}
 \frac{\dd}{\dd t}\log p_t
 &=\frac{b_0(p_t)}{p_t}
 +\sqrt{\gamma_1}\frac{(1-p_t)^{3/2}}{p_t^{1/2}}u_t^1
 -\sqrt{\gamma_2}\sqrt{p_t(1-p_t)}u_t^2 \\
 &=q_{12}\frac{(1-p_t)^2}{p_t}+(q_{11}-q_{22})(1-p_t)-q_{21}p_t \\
 &\qquad +\sqrt{\gamma_1}\frac{(1-p_t)^{3/2}}{p_t^{1/2}}u_t^1
 -\sqrt{\gamma_2}\sqrt{p_t(1-p_t)}u_t^2.
\end{align*}
Using $ay^2+by\ge -b^2/(4a)$ with $a=q_{12}$ and
$y=(1-p_t)/\sqrt{p_t}$, we obtain
\[
 q_{12}\frac{(1-p_t)^2}{p_t}
 +\sqrt{\gamma_1}\frac{(1-p_t)^{3/2}}{p_t^{1/2}}u_t^1
 \ge -C|u_t^1|^2,
\]
where $C>0$ depends only on the coefficients.

The remaining terms are bounded below by $-C(1+|u_t^2|^2)$.  Hence
\begin{equation}\label{eq:log-p-lower}
 \frac{\dd}{\dd t}\log p_t\ge -C(1+|u_t|^2).
\end{equation}

Similarly,
\begin{align*}
	\frac{\dd}{\dd t}\log(1-p_t)
	&=-\frac{b_0(p_t)}{1-p_t}
	-\sqrt{\gamma_1}\sqrt{p_t(1-p_t)}u_t^1
	+\sqrt{\gamma_2}\frac{p_t^{3/2}}{(1-p_t)^{1/2}}u_t^2
\end{align*}
leads, using $q_{21}>0$, to
\begin{equation}\label{eq:log-one-minus-p-lower}
	\frac{\dd}{\dd t}\log(1-p_t)\ge -C(1+|u_t|^2).
\end{equation}
Integrating \eqref{eq:log-p-lower} and \eqref{eq:log-one-minus-p-lower} gives the two-sided bound for $p$ in \eqref{eq:uniform-boundary-separation}, with $\delta_N$ depending only on $N,T,p_0$ and the coefficients.

On $[\delta_N,1-\delta_N]$, the coefficients in \eqref{eq:p-skeleton} are Lipschitz.  If $p_t$ and $\widetilde p_t$ are two solutions driven by the same $u_t$, then
\[
 |p_t-\widetilde p_t|\le \int_0^t C(1+|u_r|)|p_r-\widetilde p_r|\dd r,
\]
and Gronwall's lemma yields uniqueness.  Existence follows by standard localization and the preceding boundary estimates.

Once $p_t$ is fixed, \eqref{eq:s-skeleton} is linear and has the explicit solution
\begin{equation}\label{eq:s-skeleton-explicit}
 s_t=s_0\exp\left\{\int_0^t\left[A(p_r)+\eta_1(p_r)u_r^1+\eta_2(p_r)u_r^2\right]\dd r\right\}.
\end{equation}
Since $A(p)$, $\eta_1(p)$ and $\eta_2(p)$ are bounded on $[0,1]$, the $L^2$ bound on $u_t$ implies the bounds $m_N\le s_t\le M_N$.
\end{proof}

Let $\cS^N(\R^2)$ denote the closed ball
\[
 \cS^N(\R^2)=\left\{u\in L^2([0,T];\R^2):\int_0^T|u_t|^2\dd t\le N\right\},
\]
equipped with the weak topology inherited from $L^2([0,T];\R^2)$.  This topology is compact and metrizable on $\cS^N(\R^2)$.  Denote by $g^0(\int_0^\cdot u_r\dd r)$ the solution $(p_t,s_t)$ of \eqref{eq:p-skeleton}--\eqref{eq:s-skeleton}.

\begin{lemma}[Compactness of skeleton level sets]\label{lem:compact-skeleton}
For every $N<\infty$, the set
\[
 \Gamma_N=\left\{g^0\left(\int_0^\cdot u_r\dd r\right):u\in\cS^N(\R^2)\right\}
\]
is compact in $C([0,T];(0,1)\times(0,\infty))$.
\end{lemma}

\begin{proof}
By Lemma \ref{lem:skeleton-wellposed}, the family is uniformly bounded and stays in the compact set $[\delta_N,1-\delta_N]\times[m_N,M_N]$.  Moreover,
\[
 |p_t-p_r|\le C|t-r|+C\int_r^t|u_a|\dd a
 \le C|t-r|+C\sqrt{N}|t-r|^{1/2},
\]
and the same estimate for $\log s$ follows from \eqref{eq:s-skeleton-explicit}.  Hence the family is equicontinuous, and Arzela-Ascoli's criteria gives relative compactness.

It remains to check closedness.  Let $u_n\in\cS^N(\R^2)$ converge weakly in $L^2$ to $u$, and let $(p_n,s_n)=g^0(\int_0^\cdot u_{n,r}\dd r)$ converge uniformly along a subsequence to $(p,s)$.  For any continuous coefficient $F$ and any component $i$,
\[
 \int_0^t F(p_{n,r})u_{n,r}^i\dd r-\int_0^tF(p_r)u_r^i\dd r
 =\int_0^t\{F(p_{n,r})-F(p_r)\}u_{n,r}^i\dd r
 +\int_0^tF(p_r)(u_{n,r}^i-u_r^i)\dd r.
\]
The first term tends to zero by uniform convergence and the uniform $L^2$ bound, the second tends to zero by weak convergence in $L^2$.

Since $(p_n,s_n)$ is the skeleton solution associated with $u_n$, it satisfies
the integral equations corresponding to \eqref{eq:p-skeleton}--\eqref{eq:s-skeleton}
with $(p,s,u)$ replaced by $(p_n,s_n,u_n)$. Passing to the limit in these
identities shows that $(p,s)$ satisfies \eqref{eq:p-skeleton}--\eqref{eq:s-skeleton}
with control $u$. By uniqueness of the skeleton solution,
\[
(p,s)=g^0\left(\int_0^\cdot u_r\,\dd r\right).
\]
Therefore, $\Gamma_N$ is closed and hence compact.

\end{proof}

To place \eqref{eq:Peps-correct}--\eqref{eq:Seps-correct} in the abstract
weak-convergence framework of Appendix \ref{app:bdm}, we first introduce the associated
solution map. Set
$$
\mathcal W:=C_0([0,T];\mathbb R^2),
\quad
\mathcal E:=C([0,T];(0,1)\times(0,\infty)).
$$
By the strong well-posedness established in Section \ref{sec:strong-uniqueness}, applied with
$\gamma_i$ replaced by $\varepsilon\gamma_i$. For every $\varepsilon>0$
there exists a Borel measurable solution map
$$
g^\varepsilon:\mathcal W\longrightarrow\mathcal E
$$
such that
$$
(P^\varepsilon,S^\varepsilon)
=
g^\varepsilon\bigl(\sqrt{\varepsilon}\,W\bigr)
\quad\text{a.s.}
$$
We fix one such measurable version throughout this section. Thus,
$g^\varepsilon$ maps the scaled driving-noise path to the corresponding
solution path of
\eqref{eq:Peps-correct}--\eqref{eq:Seps-correct}.

Let $\cP_2^N(\R^2)$ be the collection of predictable processes $k$ with $k(\omega)\in\cS^N(\R^2)$ a.s.. 
For $u\in\mathcal P_2^N(\mathbb R^2)$, define
$$
(P^{\varepsilon,u},S^{\varepsilon,u})
:=
g^\varepsilon\left(
\sqrt{\varepsilon}\,W+\int_0^\cdot u_r\,\dd r
\right).
$$
To identify the dynamics of this process, set
$$
\widetilde W_t
:=
W_t+\frac{1}{\sqrt{\varepsilon}}
\int_0^t u_r\,\dd r.
$$
The energy bound on $u$ implies Novikov's condition for every fixed
$\varepsilon>0$. Hence, by Girsanov's theorem, there exists a probability
measure equivalent to the original one under which $\widetilde W$ is a
two-dimensional Brownian motion. Under this measure,
$$
(P^{\varepsilon,u},S^{\varepsilon,u})
=
g^\varepsilon\bigl(\sqrt{\varepsilon}\,\widetilde W\bigr)
$$
satisfies \eqref{eq:Peps-correct}--\eqref{eq:Seps-correct} with $W$
replaced by $\widetilde W$.

 Therefore, $(P^{\eps,u},S^{\eps,u})$  satisfies the following controlled system:
\begin{equation}\label{eq:controlled-Peps}
\begin{aligned}
 \dd P_t^{\eps,u}={}&\{b_0(P_t^{\eps,u})+\eps b_1(P_t^{\eps,u})
 +\alpha_1(P_t^{\eps,u})u_t^1+\alpha_2(P_t^{\eps,u})u_t^2\}\dd t \\
 &\quad +\sqrt{\eps}\{\alpha_1(P_t^{\eps,u})\dd W_t^1+\alpha_2(P_t^{\eps,u})\dd W_t^2\},
\end{aligned}
\end{equation}
\begin{equation}\label{eq:controlled-Seps}
\begin{aligned}
 \dd S_t^{\eps,u}={}&S_t^{\eps,u}\{A(P_t^{\eps,u})+
 \eta_1(P_t^{\eps,u})u_t^1+\eta_2(P_t^{\eps,u})u_t^2\}\dd t \\
 &\quad +\sqrt{\eps} S_t^{\eps,u}\{\eta_1(P_t^{\eps,u})\dd W_t^1+\eta_2(P_t^{\eps,u})\dd W_t^2\}.
\end{aligned}
\end{equation}

\begin{lemma}[Convergence of controlled processes]\label{lem:controlled-convergence}
Let $N<\infty$ and let $u^\eps\in\cP_2^N(\R^2)$ converge in distribution, as $\cS^N(\R^2)$-valued random variables with the weak topology, to $u$.  Then
\[
 g^\eps\left(\sqrt{\eps} W+\int_0^\cdot u_r^\eps\dd r\right)
 \Longrightarrow
 g^0\left(\int_0^\cdot u_r\dd r\right)
\]
in $C([0,T];[0,1]\times(0,\infty))$.
\end{lemma}

\begin{proof}
The coefficients $b_0$, $b_1$, $\alpha_i$ and $\eta_i$ are bounded on $[0,1]$.  From \eqref{eq:controlled-Peps}, for $0\le r<t\le T$,
\[
 |P_t^{\eps,u^\eps}-P_r^{\eps,u^\eps}|
 \le C|t-r|+C\int_r^t|u_a^\eps|\dd a
 +\sqrt{\eps}\left|\int_r^t\alpha_1(P_a^{\eps,u^\eps})\dd W_a^1+\int_r^t\alpha_2(P_a^{\eps,u^\eps})\dd W_a^2\right|.
\]
The deterministic terms are uniformly controlled by
$C|t-r|+C\sqrt N |t-r|^{1/2}$, and the martingale term vanishes in probability uniformly on $[0,T]$ as $\eps\to0$ by the Burkholder--Davis--Gundy inequality.  Thus $\{P^{\eps,u^\eps}\}$ is tight.

For $S^{\eps,u^\eps}$, use the logarithmic form
\begin{equation}\label{eq:log-S-controlled}
	\begin{aligned}
		\log S_t^{\eps,u^\eps}
		={}&\log s_0+
		\int_0^t\left[
		A(P_a^{\eps,u^\eps})
		+\eta_1(P_a^{\eps,u^\eps})u_a^{\eps,1}
		+\eta_2(P_a^{\eps,u^\eps})u_a^{\eps,2}
		\right]\dd a \\
		&-\frac{\eps}{2}\int_0^t\left[
		\eta_1(P_a^{\eps,u^\eps})^2
		+\eta_2(P_a^{\eps,u^\eps})^2
		\right]\dd a \\
		&+\sqrt{\eps}\int_0^t\left[
		\eta_1(P_a^{\eps,u^\eps})\dd W_a^1
		+\eta_2(P_a^{\eps,u^\eps})\dd W_a^2
		\right].
	\end{aligned}
\end{equation}
Again, boundedness of $A$ and $\eta_i$, the $L^2$ bound on $u^\eps$, and BDG imply tightness of $\log S^{\eps,u^\eps}$, and hence of $S^{\eps,u^\eps}$.

Let $(\bar p,\bar s)$ denote the limit of a convergent subsequence.  The martingale terms in \eqref{eq:controlled-Peps} and \eqref{eq:log-S-controlled} vanish in probability because their quadratic variations are of order $O(\eps)$.  Since $P^{\eps,u^\eps}\to\bar p$ uniformly along the subsequence and $u^\eps\to u$ weakly in $L^2$, the same weak-convergence argument used in the proof of Lemma \ref{lem:compact-skeleton} gives, for each continuous coefficient $F$,
\[
 \int_0^tF(P_a^{\eps,u^\eps})u_a^{\eps,i}\dd a
 \longrightarrow
 \int_0^tF(\bar p_a)u_a^i\dd a.
\]
Passing to the limit in the integral equations shows that $(\bar p,\bar s)$ satisfies \eqref{eq:p-skeleton}--\eqref{eq:s-skeleton} with control $u$.  By Lemma \ref{lem:skeleton-wellposed}, the skeleton solution is unique, so $(\bar p,\bar s)=g^0(\int_0^\cdot u_r\dd r)$.  Since every subsequential limit is the same, the full convergence follows.
\end{proof}

\begin{proof}[Proof of Theorem \ref{thm:ldp}]
Lemmas \ref{lem:compact-skeleton} and \ref{lem:controlled-convergence} verify Assumption \ref{assumption:ti} in Appendix \ref{app:bdm} for the family $(P^\eps,S^\eps)$.  Therefore Theorem \ref{thm:bdm} gives an LDP for $(P^\eps,S^\eps)$ in $C([0,T];(0,1)\times(0,\infty))$ with rate function
\begin{equation}\label{eq:ps-rate}
 J(p,s)=\inf\left\{\frac12\int_0^T|u_t|^2\dd t:
 (p,s)\text{ solves \eqref{eq:p-skeleton}--\eqref{eq:s-skeleton} with control }u\right\}.
\end{equation}
The map
\[
 \Phi:C([0,T];(0,1)\times(0,\infty))\to C([0,T];\Xo),
 \quad
 \Phi(p,s)=(ps,(1-p)s),
\]
is continuous.  Hence the contraction principle gives an LDP for $X^\eps=\Phi(P^\eps,S^\eps)$.

It remains only to identify the rate function.  If $x=\Phi(p,s)$ and $(p,s)$ solves \eqref{eq:p-skeleton}--\eqref{eq:s-skeleton} with control $u$, then direct differentiation gives
\begin{equation}\label{eq:x-skeleton}
 \dot x_t^k=\sum_{j=1}^2q_{kj}x_t^j+\sqrt{\gamma_kx_t^1x_t^2}\,u_t^k,
 \quad k=1,2.
\end{equation}
Conversely, if $x\in AC([0,T];\Xo)$ satisfies \eqref{eq:x-skeleton}, then $p=x^1/(x^1+x^2)$ and $s=x^1+x^2$ solve \eqref{eq:p-skeleton}--\eqref{eq:s-skeleton}.  Thus the control is uniquely determined by
\[
 u_t^k=\frac{\dot x_t^k-\sum_{j=1}^2q_{kj}x_t^j}
 {\sqrt{\gamma_kx_t^1x_t^2}},
 \quad k=1,2,
\]
which yields \eqref{eq:rate-function}.  The compactness already proved in Lemma \ref{lem:compact-skeleton} shows that this rate function is good.
\end{proof}

\appendix
\section{A large deviation framework}
\label{app:bdm}

For the convenience of the reader, we recall the definition of the LDP and a weak-convergence criterion due to Budhiraja, Dupuis and Maroulas \cite{BDM}.

\begin{definition}[Large deviation principle]\label{def:ldp}
Let $E$ be a Polish space.  A family $\{u^\eps:\eps>0\}$ of $E$-valued random variables is said to satisfy the LDP as $\eps\to0$ with rate function $I:E\to[0,\infty]$ if $I$ is lower semicontinuous, the level set $\{x\in E:I(x)\le c\}$ is compact for every $c<\infty$, and
\[
 \liminf_{\eps\to0}\eps\log\Prob(u^\eps\in G)\ge -\inf_{x\in G}I(x)
\]
for every open set $G\subset E$, while
\[
 \limsup_{\eps\to0}\eps\log\Prob(u^\eps\in F)\le -\inf_{x\in F}I(x)
\]
for every closed set $F\subset E$.
\end{definition}

Let $B$ be a $d$-dimensional Brownian motion and let $\mathbb S=C([0,T];\R^d)$.  For each $\eps>0$, let $g^\eps:\mathbb S\to E$ be measurable and define
\[
 u^\eps=g^\eps(\sqrt{\eps} B).
\]
For $N<\infty$ set
\[
 \cS^N(\R^d)=\left\{k\in L^2([0,T];\R^d):\int_0^T|k_s|^2\dd s\le N\right\},
\]
equipped with the weak topology inherited from $L^2([0,T];\R^d)$.  Let $\cP_2^N(\R^d)$ be the collection of predictable processes $k$ with $k(\omega)\in\cS^N(\R^d)$ a.s..

\begin{assumption}\label{assumption:ti}
There exists a measurable map $g^0:\mathbb S\to E$ such that:
\begin{enumerate}[label=\textup{(\roman*)}]
\item for every $N<\infty$,
\[
 \Gamma_N=\left\{g^0\left(\int_0^\cdot k_s\dd s\right):k\in\cS^N(\R^d)\right\}
\]
is a compact subset of $E$;
\item if $k^\eps\in\cP_2^N(\R^d)$ converges in distribution, as an $\cS^N(\R^d)$-valued random variable with the weak topology, to $k$, then
\[
 g^\eps\left(\sqrt{\eps} B+\int_0^\cdot k_s^\eps\dd s\right)
 \rightarrow
 g^0\left(\int_0^\cdot k_s\dd s\right)
\]
in distribution as $\eps\to0$.
\end{enumerate}
\end{assumption}

\begin{theorem}[Budhiraja--Dupuis--Maroulas]\label{thm:bdm}
Suppose that Assumption \ref{assumption:ti} holds.  Then $\{u^\eps\}$ satisfies the LDP on $E$ with good rate function
\[
 I(f)=\inf\left\{\frac12\int_0^T|k_s|^2\dd s:
 k\in L^2([0,T];\R^d),\quad
 f=g^0\left(\int_0^\cdot k_s\dd s\right)\right\},
\]
where the infimum over the empty set is understood to be $\infty$.
\end{theorem}


\begin{thebibliography}{99}

\bibitem{ABBP} Athreya, S. R., Barlow, M. T., Bass, R. F. and Perkins, E. A. (2002). Degenerate stochastic differential equations and super-Markov chains. \emph{Probab. Theory Related Fields} \textbf{123}, 484--520.

\bibitem{BC} Baldi, P. and Caramellino, L. (2011). General Freidlin-Wentzell large deviations and positive diffusions. \emph{Statist. Probab. Lett.} \textbf{81}, no. 8, 1218--1229.

\bibitem{BP} Bass, R. F. and Perkins, E. A. (2003). Degenerate stochastic differential equations with H\"older continuous coefficients and super-Markov chains. \emph{Trans. Amer. Math. Soc.} \textbf{355}, no. 1, 373--405.

\bibitem{BP2} Bass, R. F. and Perkins, E. A. (2004). Countable systems of degenerate stochastic differential equations with applications to super-Markov chains. \emph{Electron. J. Probab.} \textbf{9}, 634--673.

\bibitem{BDM} Budhiraja, A., Dupuis, P. and Maroulas, V. (2008). Large deviations for infinite dimensional stochastic dynamical systems. \emph{Ann. Probab.} \textbf{36}, no. 4, 1390--1420.

\bibitem{BMP} Burdzy, K., Mueller, C. and Perkins, E. A. (2010). Nonuniqueness for nonnegative solutions of parabolic stochastic partial differential equations. \emph{Illinois J. Math.} \textbf{54}, no. 4, 1481--1507.

\bibitem{Daw} Dawson, D. (1975). Stochastic evolution equations and related measure processes. \emph{J. Multivariate Anal.} \textbf{5}, 1--52.

\bibitem{DF} Dawson, D. and Fleischmann, K. (1997). A continuous super-Brownian motion in a super-Brownian medium. \emph{J. Theoret. Probab.} \textbf{10}, no. 1, 213--276.

\bibitem{DFX} Dawson, D. A., Fleischmann, K. and Xiong, J. (2005). Strong uniqueness for cyclically symbiotic branching diffusions. \emph{Statist. Probab. Lett.} \textbf{73}, 251--257.

\bibitem{DP} Dawson, D. A. and Perkins, E. A. (1998). Long-time behavior and coexistence in a mutually catalytic branching model. \emph{Ann. Probab.} \textbf{26}, no. 3, 1088--1138.

\bibitem{DP2} Dawson, D. A. and Perkins, E. A. (2006). On the uniqueness problem for catalytic branching networks and other singular diffusions. \emph{Illinois J. Math.} \textbf{50}, no. 1--4, 323--383.

\bibitem{Deb} DeBlassie, D. (2004). Uniqueness for diffusions degenerating at the boundary of a smooth bounded set. \emph{Ann. Probab.} \textbf{32}, no. 4, 3167--3190.




\bibitem{FX} Fleischmann, K. and Xiong, J. (2001). A cyclically catalytic super-Brownian motion. \emph{Ann. Probab.} \textbf{29}, no. 2, 820--861.

\bibitem{FreidlinWentzell}
Freidlin, M. I. and Wentzell, A. D. (2012).
Random perturbations of dynamical systems.
3rd ed., Grundlehren der Mathematischen Wissenschaften, Vol. 260,
Springer, Berlin, Heidelberg.

\bibitem{He} He, H. (2009). Strong uniqueness for a class of singular SDEs for catalytic branching diffusions. \emph{Statist. Probab. Lett.} \textbf{79}, no. 2, 182--187.

\bibitem{iw} Ikeda, N. and Watanabe, S. (1977). A comparison theorem for solutions of stochastic differential equations and its applications. \emph{Osaka J. Math.} \textbf{14}, no. 3, 619--633.

\bibitem{KS} Konno, N. and Shiga, T. (1988). Stochastic partial differential equations for some measure-valued diffusions. \emph{Probab. Theory Related Fields} \textbf{79}, 201--225.

\bibitem{MMP} Mueller, C., Mytnik, L. and Perkins, E. (2014). Nonuniqueness for a parabolic SPDE with $3/4-\varepsilon$-H\"older diffusion coefficients. \emph{Ann. Probab.} \textbf{42}, no. 5, 2032--2112.

\bibitem{Myt} Mytnik, L. (1998). Uniqueness for a mutually catalytic branching model. \emph{Probab. Theory Related Fields} \textbf{112}, no. 2, 245--253.

\bibitem{MP} Mytnik, L. and Perkins, E. A. (2011). Pathwise uniqueness for stochastic heat equations with H\"older continuous coefficients: the white noise case. \emph{Probab. Theory Related Fields} \textbf{149}, no. 1--2, 1--96.

\bibitem{MPS} Mytnik, L., Perkins, E. A. and Sturm, A. (2006). On pathwise uniqueness for stochastic heat equations with non-Lipschitz coefficients. \emph{Ann. Probab.} \textbf{34}, no. 5, 1910--1959.

\bibitem{Rei} Reimers, M. (1989). One-dimensional stochastic differential equations and the branching measure diffusion. \emph{Probab. Theory Related Fields} \textbf{81}, 319--340.

\bibitem{Swa2} Swart, J. M. (2002). Pathwise uniqueness for an SDE with non-Lipschitz coefficients. \emph{Stochastic Process. Appl.} \textbf{98}, no. 1, 131--149.

\bibitem{Wat} Watanabe, S. (1968). A limit theorem of branching processes and continuous state branching processes. \emph{J. Math. Kyoto Univ.} \textbf{8}, 141--167.

\bibitem{Xio} Xiong, J. (2013). Super-Brownian motion as the unique strong solution to an SPDE. \emph{Ann. Probab.} \textbf{41}, no. 2, 1030--1054.

\end{thebibliography}
\end{document}